\documentclass{article}
\usepackage[pdftex]{hyperref}
\usepackage{tikz-cd, url}

\hypersetup{
  colorlinks   = true, 
  urlcolor     = gray, 
  linkcolor    = red, 
  citecolor   =  blue
}
\usepackage{amsmath,amsfonts,amsthm,amssymb,graphicx,xcolor}
\usepackage[all,2cell,ps]{xy}

\theoremstyle{plain}
\newtheorem{thm}{Theorem}[section]
\newtheorem{lem}[thm]{Lemma}

\newtheorem{conj}[thm]{Conjecture}

\theoremstyle{definition}

\newtheorem{rem}[thm]{Remark}

\DeclareMathOperator{\Alb}{Alb}

\usepackage{mathtools}

\DeclarePairedDelimiter\abs{\lvert}{\rvert}
\DeclarePairedDelimiter\norm{\lVert}{\rVert}

\makeatletter
\let\oldabs\abs
\def\abs{\@ifstar{\oldabs}{\oldabs*}}
\let\oldnorm\norm
\def\norm{\@ifstar{\oldnorm}{\oldnorm*}}
\makeatother

\title{Singer Conjecture for Varieties with Semismall Albanese Map and Residually Finite Fundamental Group}
 \author{\small{Luca F. Di Cerbo}\footnote{Supported in part by NSF grant DMS-2104662} \\ \scriptsize{University of Florida} \\ \footnotesize{\textsf{ldicerbo@ufl.edu}} \and 
\small{Luigi Lombardi}\footnote{Partially supported by GNSAGA-INDAM, 
PRIN 2020: ``Curves, Ricci flat varieties and their interactions'', 
and PRIN 2022: ``Synplectic varieties: their interplay with Fano manifolds and derived categories''.} \\ 
\scriptsize{Università degli Studi di Milano Statale}\\ \footnotesize{\textsf{luigi.lombardi@unimi.it}}}
\date{}
	
\begin{document}

\maketitle

\begin{abstract}
 We prove the Singer conjecture for varieties with semismall Albanese map and residually finite fundamental group.
 
\end{abstract}
\vspace{8cm}
\tableofcontents\quad\\

\vspace{1cm}

\section{Introduction and Main Results}

An important problem in modern geometry and topology is a conjecture of Singer concerning the $L^2$-Betti numbers of an aspherical closed manifold. 

\begin{conj}[Singer Conjecture]\label{Singer}
	If $X$ is a closed aspherical manifold of real dimension $2n$, then the $L^2$-Betti numbers are:
	\begin{equation*}
	b^{(2)}_{k}(X; \widetilde{X})=\begin{cases} 
	(-1)^n \chi_{\rm top}(X)  &\mbox{if }\quad    k =  n  \\
	0 & \mbox{if }\quad   k  \neq n
	\end{cases}
	\end{equation*}
	where $\pi \colon \widetilde{X}\rightarrow X$ is the topological universal cover of $X$.
\end{conj}

This conjecture was inspired by Atiyah's work \cite{Atiyah} on the $L^2$-index theorem for coverings.
Moreover  its resolution in the affirmative would   settle an old problem of Hopf regarding  the sign of the Euler characteristic of aspherical manifolds. 

\begin{conj}[Hopf Conjecture]\label{Hopf}
	If $X$ is a closed aspherical manifold of real dimension $2n$, then:
	\begin{equation*}
	(-1)^n\chi_{\rm top}(X)\geq 0.
	\end{equation*}
	
\end{conj}

Conjectures \ref{Singer} and \ref{Hopf} played an important role in the development of modern differential geometry, geometric topology, and algebraic geometric. Indeed, they figure prominently in Yau's influential list of problems in geometry, see \cite[Section VII, Problem 10]{SchoenY} and \cite[Section IX, Problem 39]{SchoenY} for a variation of Conjecture \ref{Singer} in terms of normalized Betti numbers of Galois covers. We refer to  L\"uck's book \cite{Luck02} for a comprehensive introduction to this circle of ideas, for the definition of $L^2$-Betti numbers, and for a detailed historical account on the Singer conjecture (\emph{cf}. \cite[\S 11]{Luck02}). 

Despite being many decades old, these problems continue to be at the center of a substantial amount of research activity. Many researchers are currently addressing these conjectures (see for instance \cite{Bergeron2}, \sloppy \cite{DCL19a}, \cite{AOS}, \cite{LMW17a}, \cite{LMW17b}, \cite{DS22}) using diverse techniques coming from algebraic geometry, geometric analysis, and geometric topology. For classical papers on this problem, the interested reader may refer to \cite{Dod79}, \cite{DX84}, \cite{Gro}, \cite{Luck}, \cite{JZ}, just to name a few. 

In this paper, we contribute to the study of Conjecture \ref{Singer} within the realm of smooth projective varieties. More precisely, we prove Singer conjecture for smooth projective irregular varieties with semismall Albanese map and residually finite fundamental group.

\begin{thm}\label{main1}
	Let $X$ be a smooth projective variety of complex dimension $n$  and let $\widetilde X$ be the topological universal cover. 
	If the Albanese map of $X$ is semismall and $\pi_1(X)$ is residually finite, then the $L^2$-Betti numbers are:
	\begin{equation*}
	b^{(2)}_k (X; \widetilde{X})=\begin{cases} 
	(-1)^n \chi_{\rm top}(X)  &\mbox{if }\quad    k =  n  \\
	0 & \mbox{if }\quad   k  \neq n.
	\end{cases}
	\end{equation*}
\end{thm}

We refer to Section \ref{Proofs} for the details of the proof of Theorem \ref{main1}. Interestingly, our result also covers many instances of projective varieties that are not necessarily aspherical, so in many ways, 
we extend the scope of the original statement of  Singer conjecture. Moreover, varieties with semismall Albanese map need not be K\"ahler hyperbolic in the sense of Gromov \cite{Gro}. In this regard, our result complements and extends Gromov's vanishing theorem and its subsequent extension by Jost-Zuo \cite{JZ}. Note that Gromov and Jost-Zuo's theorems still represent the state of the art of Conjecture \ref{Singer} for complex manifolds. While our result does not recover their statements completely, 
it covers many other cases that are currently out of reach for the analytical techniques of Gromov and Jost-Zuo. 
Moreover, in a sense, it  is not  so unreasonable to ask whether most of 
aspherical irregular projective varieties of maximal Albanese dimension have semismall Albanese map. 
Indeed, it is tantalizing to wonder when for an aspherical irregular projective variety $X$
there  exists a variety $Y$ with semismall Albanese map and with the same universal cover of $X$.
Alternatively, one could ask when a finite unramified cover $Y$ of $X$ has a semismall Albanese map.
 Notice that, because of Lemma \ref{etale}, the property of having a semismall Albanese map is preserved under finite unramified covers. 
  Clearly, these are far reaching questions of topological flavor. We hope to discuss and place such questions in a more general framework elsewhere.

Next, if we remove the residually finiteness assumption on $\pi_{1}(X)$, we have a similar statement for the $L^2$-Betti numbers computed with respect to the algebraic universal cover $\hat{\pi}\colon\hat{X}\to X$. 

\begin{thm}\label{main2}
	Let $X$ be a smooth projective variety of complex dimension $n$  and let $\hat{X}$ be the algebraic universal cover. 
    If the Albanese map of $X$ is semismall, then the $L^2$-Betti numbers are:
	\begin{equation*}
	b^{(2)}_k (X; \hat{X})=\begin{cases} 
	(-1)^n \chi_{\rm top}(X)  &\mbox{if }\quad    k =  n  \\
	0 & \mbox{if }\quad   k  \neq n.
	\end{cases}
	\end{equation*}
\end{thm}

The proof of Theorem \ref{main2} is outlined in Section \ref{Proofs}, and it is very similar to the proof of Theorem \ref{main1}.

We conclude the paper with some applications of our theorems to the $L^2$-cohomology of the topological (algebraic) universal covers of varieties with semismall Albanese maps. We refer to Section \ref{Applications} for the precise statements and the details of the proofs.

\noindent\textbf{Acknowledgments}. 
We are grateful to the referee for his/her/their comments and improving the exposition of the paper.
The first named author thanks  Roberto Svaldi for valuable feedback and comments. He also thanks the Mathematics Department of the University of Milan for the invitation to present this research, for support, and for the nice working environment during his visit in the Spring of 2023. The second named author thanks Alice Garbagnati  for answering to all his questions, and the Mathematics Department of the University of Florida for the optimal working environment provided during his visit in the Spring of 2023.

\section{Proofs of the Main Results}\label{Proofs}

Given a manifold $X$  whose  fundamental group $\Gamma\stackrel{{\rm def}}{=}\pi_1(X)$ is residually finite,  we consider a sequence of nested, normal, finite index subgroups $\{\Gamma_i\}_{i=1}^{\infty}$ of $\Gamma$ such that $\cap_{i=1}^{\infty}\Gamma_i$ is the identity element. Such  sequence is usually called a \emph{cofinal filtration} of $\Gamma$. Define $\pi_i\colon X_i\rightarrow X$ as the finite  regular cover of $X$ associated to $\Gamma_i$. The main result of \cite{Luck} implies that
\begin{align}\label{LIMIT}
\lim_{i \to \infty}\frac{b_{k}(X_i)}{\deg \pi_i } \; = \; b^{(2)}_k (X; \widetilde{X}),
\end{align}
where $b_k(X_i)$ denotes the $k$-th Betti number of $X_i$, and  $b^{(2)}_k (X; \widetilde{X})$ is the $L^2$-Betti number of $X$ computed with respect to the universal cover $\widetilde{X}$. Notice that  this result implies that the limit in   \eqref{LIMIT} always exists and it is \emph{independent} of the cofinal filtration. We refer to the ratio $b_k(X_i)/\deg \pi_i$ as the \emph{normalized} $k$-Betti number of the cover $\pi_i\colon  X_i \rightarrow X$. Thus, outside the middle dimension, the Singer conjecture is equivalent to the sub-degree growth of Betti numbers along a tower of covers associated to a cofinal filtration.

We now turn to details and present our main results. Let $X$ be an irregular smooth projective complex variety of dimension $n$, and let $a_X\colon X \to \Alb(X)$ be its Albanese map. The Albanese torus $\Alb(X)$ is an abelian variety of dimension $g=h^{1, 0}(X)$. Recall that a projective variety is called irregular if $g>0$, that is, if and only if the first Betti number of $X$ is non-zero. 
Define the varieties $\Alb(X)^{\ell} = \{ y\in \Alb(X) \, \big | \, \dim a_X^{-1}(y) = \ell \}$ together with the \emph{defect of 
semismallness} of the Albanese map
$$\delta (a_X) = \max_{ \big\{ \ell \geq 0 \, | \, \Alb(X)^{\ell} \neq \emptyset \big\} }\{ 2 \ell + \dim \Alb(X)^{\ell} - \dim X\}.$$
Then $\delta(a_X)\geq 0$ and if $\delta(a_X)=0$ we say that $a_X$ is \emph{semismall}.
%
%
If $a_X$ is semismall, then it is generically finite onto its image, but the converse does not hold in general. 
For instance, the Albanese map of the blow-up of an abelian variety along a smooth subvariety of codimension $c\geq 2$ 
is semismall if and only if $c= 2$.

\begin{lem}\label{etale}
Let $X$ be a smooth projective variety and let
 $f\colon Y\to X$ be a finite unramified cover. Then  the inequality $\delta(a_X ) \geq \delta(a_Y)$ holds. In particular, if $a_X$ is
 semismall, then also $a_Y$ is semismall.
\end{lem}

\begin{proof}
Let $a_f \colon \Alb(X) \to \Alb(Y)$ be the induced morphism induced by the universal property of the Albanese variety so that the 
following diagram commutes
\begin{equation}\notag
\centerline{ \xymatrix@=32pt{
Y\ar[d]^{f} \ar[r]^{a_Y\,\,\,\,\,\,\,} & \Alb(Y) \ar[d]^{a_f}  \\
X  \ar[r]^{a_X\,\,\,\,\,\,\,}  & \Alb(X).\\}} 
\end{equation}
We notice that $a_f$ is surjective since $f$ is so. 
Let $p\in a_X(X)$ and $q\in a_f^{-1}(p)$. There is an inequality
$$ \dim a_X^{-1}( p ) = \dim (a_X \circ f)^{-1} (p) = \dim (a_f \circ a_Y)^{-1} (p) \geq \dim a_Y^{-1}(q)$$ 
showing that the fiber dimension of the Albanese map does not increase in finite covers.

\end{proof}

%

\begin{rem}
The converse of Lemma \ref{etale} does not hold in general. 
For instance  one can consider a bielliptic surface and the covering abelian surface associated to the canonical divisor. 
For more details see \cite[Chapter VI]{Bea}. In particular bielliptic surfaces are aspherical surfaces that admit a finite unramified cover
with semismall Albanese map.
\end{rem}

%

We can now prove our main theorem.

\begin{proof}[Proof of Theorem \ref{main1}]

Fix an integer $k\neq n$ and 
let us consider a cofinal tower $\tau_{\ell} \colon X_{\ell} \to X$ of  $\pi_1(X)$:
$$ X \leftarrow X_1 \leftarrow X_2 \leftarrow \cdots \leftarrow X_{\ell}\leftarrow \cdots$$
By Lemma \ref{etale} the Albanese maps of the varieties $X_{\ell}$ are semismall.
We will construct a new cofinal tower  of $\pi_1(X)$  with a control on the $b_k$
by means of  covers induced by multiplication maps on the Albanese varieties as in \cite[Corollary 1.2]{DCL19a}.

Let $\psi_1 \colon Y_1\to X_1$ be the unramified cover constructed as the pullback of a multiplication map
$\mu_{d} \colon \Alb(X_1) \to \Alb(X_1)$ ($d\gg 1$)   such that 
$$\frac{b_k(Y_1)}{\deg \psi_1} \; \leq\; 1$$ (\emph{cf}. \cite[Corollary 1.2]{DCL19a}).
Denote by  $\varphi_1 = (\tau_1 \circ \psi_1 )\colon Y_1 \to X$   the natural composition map. 
We easily check that 
$$\frac{b_k(Y_1)}{\deg \varphi_1} \; \leq \;  \frac{b_k(Y_1)}{\deg \psi_1} \; \leq\;  1.$$
Now let $Z_1$ be the 
pullback of the cover $X_2 \to X_1$ along the map $\psi_1$. 
We can repeat the previous procedure in order to  construct an unramified cover 
$\psi_2 \colon Y_2 \to Z_1$ such that 
$$\frac{b_k(Y_2)}{\deg \psi_2} \; \leq \;  \frac{1}{2}.$$
By setting $Z_2$ for the pullback of the cover $X_3 \to X_2$  along the composition $Y_2\to Z_1\to X_2$, 
we can  reiterate the process and 
construct the following commutative diagram
$$
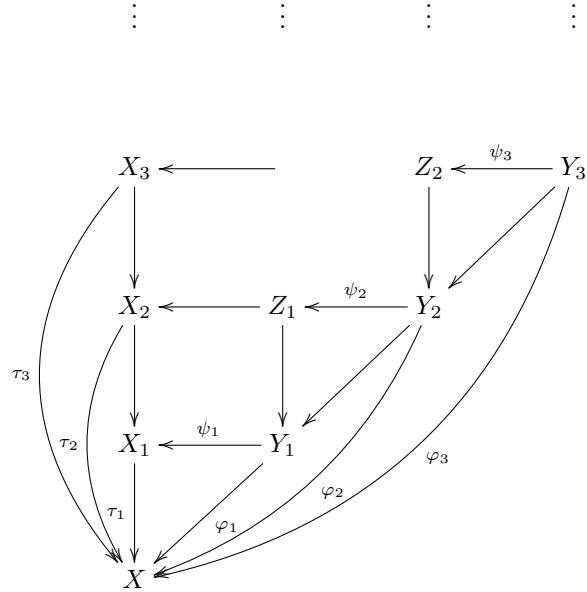
\centerline{ \xymatrix@=38pt{
\vdots & \vdots & \vdots & \vdots \\
 X_3 \ar[d] \ar@/_3pc/[ddd]_{\tau_3} & \ar[l] &  Z_2 \ar[d]  & Y_3\ar[l]_{\psi_3} \ar[dl] \ar@/^3pc/[dddlll]^{\varphi_3}\\
 X_2 \ar[d]\ar@/_1.5pc/[dd]_{\tau_2} &  Z_1 \ar[d]\ar[l] & Y_2\ar[l]_{\psi_2}\ar[dl] \ar@/^1.5pc/[ddll]^{\varphi_2}\\
 X_1\ar[d]_{\tau_1} & Y_1\ar[l]_{\psi_1}\ar[dl]^{\varphi_1}  \\
 X\\}}
 $$
%
such that 
$$\frac{b_k(Y_{\ell})}{\deg \psi_{\ell}}\leq \frac{1}{\ell}$$
  for all $\ell\geq 1$. Denote now by $\varphi_\ell \colon Y_{\ell} \to X$ the natural composition map defined as in the previous commutative diagram.
Since the covers $\{\tau_{\ell}\}_{\ell =1}^{\infty}$ form a cofinal filtration of $\pi_1(X)$,   also the sequence 
$\{\varphi_{\ell}\}_{\ell =1}^{\infty}$ forms a cofinal filtration of $\pi_1(X)$.
Moreover for any $\ell\geq 1$ we obtain
$$\frac{b_k(Y_{\ell})}{ \deg \varphi_{\ell}} \;  \leq \;  \frac{b_k(Y_{\ell})}{\deg \psi_{\ell}}\;  \leq \;  \frac{1}{\ell}.$$ 
The conclusion follows by applying L\"uck's approximation theorem \cite{Luck} along the covers $\varphi_{\ell}$.
\end{proof}

The proof of Theorem \ref{main2} is  completely analogous. Indeed, by definition of $\hat{X}$ one can construct a zig-zag sequence as in the proof of Theorem \ref{main1} converging to $\hat{X}$. For more details about the algebraic universal cover and how to generate sequences of covers that converge to it, we refer to \cite[Theorem 2.7]{DD19}. 

\section{Application: Existence of $L^2$-Integrable Harmonic Forms}\label{Applications}
 
In this section, we collect a few applications of our results to the $L^2$-cohomology of the (algebraic) universal cover of a smooth projective variety such that $a_X\colon X \to \Alb(X)$ is semismall. Let $\pi \colon \widetilde{X}\to X$ be the topological universal cover, and let $\hat{\pi} \colon \hat{X}\to X$ be the algebraic universal cover. We start by defining $L^2$-cohomology. Given any Riemannian metric on $X$, consider its pull-back to $\widetilde{X}$. By using the pulled back metric, define the Hilbert space of smooth $L^2$-integrable harmonic  $k$-forms 
\[
\mathcal{H}^k_{(2)} (\widetilde{X}) \; = \; \big\{ \, \omega\in\Omega^k (\widetilde{X}) \;  | \;  \Delta_d \,  \omega = 0, \; \int_{\widetilde{X}} \omega \wedge *\omega  <\infty \big\}
\]
where $*$ is the Hodge-star operator and $\Delta_d = d d^* + d^* d$ is the Hodge-Laplacian operator. These spaces do not depend on the given metric considered on $X$ (\emph{cf.} \cite{Atiyah})  and  compute the $L^2$-cohomology of $\widetilde{X}$. 

\begin{thm}\label{corl1}
Let $X$ be a smooth projective variety of complex dimension $n$ such that the Albanese map  $a_X\colon X \to \Alb(X)$ is semismall and $\pi_1(X)$ is residually finite.
If $\chi_{\rm top}(X) \neq 0$, then there exists a nontrivial harmonic $L^2$-integrable $n$-form on the topological universal cover $\widetilde X$.
\end{thm}

\begin{proof}
The proof is a combination of Theorem \ref{main1}, L\"uck approximation theorem, and \cite[pp. 6-7]{JZ}.
\end{proof}

If $\pi_1(X)$ is not residually finite, by Theorem \ref{main2} we have an analogous result for the $L^2$-cohomology of the algebraic universal cover $\hat{X}$.

\begin{thm}\label{corl2}
Let $X$ be a smooth projective variety of complex dimension $n$  such that the Albanese map $a_X\colon X \to \Alb(X)$ is semismall.
If $\chi_{\rm top}(X) \neq 0$, then there exists a nontrivial harmonic $L^2$-integrable $n$-form on the algebraic universal cover $\hat{X}$. 
\end{thm}
%

%
\bibliography{biblsinger}

\end{document}